\documentclass[reqno,11pt]{amsart}
\usepackage{amsmath,amsfonts,amssymb,amsthm,color,hyperref,enumerate,mathrsfs,mathtools,todonotes,tikz-cd}
\usepackage[all]{xy}
\usepackage[top=2.8cm,bottom=2.8cm,left=2.5cm,right=2.5cm]{geometry}
\usetikzlibrary{chains,decorations.pathreplacing,backgrounds,decorations.markings,babel,positioning,patterns,shapes}
\usepackage{accents}

\usepackage[T1]{fontenc}
\usepackage{bbm}

\newcounter{subeqn} %
\makeatletter
\@addtoreset{subeqn}{equation}
\makeatother

\tikzset{double line with arrow/.style args={#1,#2}{decorate,decoration={markings,%
mark=at position 0 with {\coordinate (ta-base-1) at (0,1pt);
\coordinate (ta-base-2) at (0,-1pt);},
mark=at position 1 with {\draw[#1] (ta-base-1) -- (0,1pt);
\draw[#2] (ta-base-2) -- (0,-1pt);
}}}}

\tikzset{
    labl/.style={anchor=south, rotate=90, inner sep=.5mm}
}

\newtheorem{Theorem}{Theorem}
\newtheorem{Proposition}[Theorem]{Proposition}
\newtheorem{Lemma}[Theorem]{Lemma}

\newtheorem{Corollary}[Theorem]{Corollary}

\newtheorem{Remark}[Theorem]{Remark}

\newcommand{\nc}{\newcommand}
\newcommand{\renc}{\renewcommand}

\nc{\fa}{\mathfrak a}
\nc{\fb}{\mathfrak b}
\nc{\fg}{\mathfrak g}
\nc{\fk}{\mathfrak k}
\nc{\fh}{\mathfrak h}
\nc{\ft}{\mathfrak t}
\nc{\fw}{\mathfrak w}
\nc{\fM}{\mathfrak M}
\nc{\CC}{\mathbb{C}}
\nc{\GG}{\mathbb{G}}
\nc{\KK}{\mathbb{K}}
\nc{\NN}{\mathbb{N}}
\nc{\PP}{\mathbb{P}}
\nc{\QQ}{\mathbb{Q}}
\nc{\RR}{\mathbb{R}}
\nc{\ZZ}{\mathbb{Z}}
\nc{\cA}{\mathcal{A}}
\nc{\cB}{\mathcal{B}}
\nc{\cC}{\mathcal{C}}
\nc{\cD}{\mathcal{D}}
\nc{\cE}{\mathcal{E}}
\nc{\cF}{\mathcal{F}}
\nc{\cG}{\mathcal{G}}
\nc{\cI}{\mathcal{I}}
\nc{\cJ}{\mathcal{J}}
\nc{\cK}{\mathcal{K}}
\nc{\cL}{\mathcal{L}}
\nc{\cM}{\mathcal{M}}
\nc{\cN}{\mathcal{N}}
\nc{\cO}{\mathcal{O}}
\nc{\cP}{\mathcal{P}}
\nc{\cQ}{\mathcal{Q}}
\nc{\cR}{\mathcal{R}}
\nc{\cS}{\mathcal{S}}
\nc{\cT}{\mathcal{T}}
\nc{\cU}{\mathcal{U}}
\nc{\cV}{\mathcal{V}}
\nc{\cW}{\mathcal{W}}
\nc{\cX}{\mathcal{X}}
\nc{\cY}{\mathcal{Y}}
\nc{\cZ}{\mathcal{Z}}
\nc{\ba}{\mathbf{a}}
\nc{\bb}{\mathbf{b}}
\nc{\bc}{\mathbf{c}}
\nc{\bd}{\mathbf{d}}
\nc{\bg}{\mathbf{g}}
\nc{\bi}{\mathbf{i}}
\nc{\bj}{\mathbf j}
\nc{\bk}{\mathbf k}
\nc{\bm}{\mathbf{m}}
\nc{\bs}{\mathbf{s}}
\nc{\bt}{\mathbf{t}}
\nc{\bv}{\mathbf{v}}
\nc{\bw}{\mathbf{w}}
\nc{\bz}{\mathbf{z}}
\nc{\bA}{\mathbf A}
\nc{\bB}{\mathbf B}
\nc{\bF}{\mathbf F}
\nc{\bG}{\mathbf G}
\nc{\bL}{\mathbf{L}}
\nc{\bM}{\mathbf{M}}
\nc{\bN}{\mathbf{N}}
\nc{\bP}{\mathbf{P}}
\nc{\bR}{\mathbf{R}}
\nc{\bS}{\mathbf{S}}
\nc{\bT}{\mathbf{T}}
\nc{\bU}{\mathbf{U}}
\nc{\bX}{\mathbf{X}}
\nc{\sB}{\mathscr{B}}
\nc{\sC}{\mathscr{C}}
\nc{\Ann}{\operatorname{Ann}}
\nc{\Aut}{\operatorname{Aut}}
\nc{\Coker}{\operatorname{Coker}}
\nc{\Der}{\operatorname{Der}}
\nc{\End}{\operatorname{End}}
\nc{\Ext}{\operatorname{Ext}}
\nc{\fmod}{\operatorname{--fmod}}
\nc{\gr}{\operatorname{gr}}
\nc{\Hom}{\operatorname{Hom}}
\renc{\Im}{\operatorname{Im}}
\nc{\Ind}{\operatorname{Ind}}
\nc{\Ker}{\operatorname{Ker}}
\nc{\MaxSpec}{\operatorname{MaxSpec}}
\nc{\Mod}{\operatorname{Mod}}
\nc{\op}{\operatorname{op}}
\nc{\Proj}{\operatorname{Proj}}
\nc{\Rep}{\operatorname{Rep}}
\nc{\Res}{\operatorname{Res}}
\nc{\Span}{\operatorname{Span}}
\nc{\Spec}{\operatorname{Spec}}
\nc{\Sym}{\operatorname{Sym}}
\nc{\Tor}{\operatorname{Tor}}
\nc{\tr}{\operatorname{tr}}
\nc{\val}{\operatorname{val}}
\nc{\wt}{\operatorname{wt}}
\nc{\arxiv}[1]{\href{http://arxiv.org/abs/#1}{\tt arXiv:\nolinkurl{#1}}}
\nc{\Cartan}{\CC[H_\bullet^{(\bullet)}]}
\nc{\con}{\sim}
\nc{\diam}{\diamond}
\nc{\eps}{\varepsilon}
\nc{\fsl}{\mathfrak{sl}}
\nc{\flavour}{G_W}
\nc{\gauge}{G_V}
\nc{\GL}{\operatorname{GL}}
\nc{\Gr}{\mathsf{Gr}}
\nc{\Grml}{\Gr_\mu^{\overline{\lambda}}}
\nc{\hh}{\hslash}
\nc{\id}{\operatorname{id}}
\nc{\leftexp}[2]{\vphantom{#2}^{#1} #2}
\nc{\Lie}{\operatorname{Lie}}
\nc{\nilHecke}{\mathcal{NH}_\bG}
\nc{\sslash}{/\mkern-6mu/}
\nc{\ssslash}{/\mkern-6mu/\mkern-6mu/}
\nc{\Stab}{\mathsf S}
\nc{\hooklongrightarrow}{\lhook\joinrel\longrightarrow}
\nc{\hooklongleftarrow}{\longleftarrow\joinrel\defparamok}
\nc{\twoheadlongrightarrow}{\relbar\joinrel\twoheadrightarrow}
\nc{\twoheadlongleftarrow}{\twoheadleftarrow\joinrel\relbar}
\nc{\acom}[1]{\todo[inline,color=green!20]{ Alex: #1 }}

\nc{\CB}{\cA^{\mathsf{sph}}}
\nc{\ICB}{\cA}
\nc{\ACB}{\cA^{\mathsf{ab}}}
\nc{\tCB}{\widetilde{\ICB}^{\mathsf{sph}}}
\nc{\tICB}{\widetilde{\ICB}}
\nc{\CBone}{\CB_{\hbar=1}}
\nc{\ICBone}{\ICB_{\hbar=1}}
\nc{\ACBone}{\ACB_{\hbar=1}}
\nc{\GK}{\bG_{\cK}}
\nc{\GO}{\bG_{\cO}}
\nc{\TK}{\bT((z))}
\nc{\TO}{\bT[[z]]}
\nc{\Iwa}{\cI}
\nc{\NK}{\bN_{\cK}}
\nc{\NO}{\bN_{\cO}}
\nc{\la}{\lambda}
\nc{\Ya}{{\reflectbox{\rm R}}}
\nc{\Yml}{Y_\mu^\lambda}
\nc{\FYml}{FY_\mu^\lambda}
\nc{\barX}{\bar{X}}
\nc{\barQ}{\bar{Q}}
\nc{\barP}{\bar{P}}
\nc{\excise}[1]{}
\nc{\Pol}{\mathsf{Pol}}
\nc{\PolKLR}{\mathsf{Pol}}
\nc{\yz}{z}
\nc{\YZ}{Z}

\nc{\codim}{\operatorname{codim}}
\nc{\cochar}{X_\ast(\bF)}

\nc{\ev}{\operatorname{ev}}
\nc{\basW}{\eps}
\nc{\basV}{e}
\nc{\DEF}{=}
\nc{\idem}{\mathsf e}
\nc{\idemCB}{\idem}
\nc{\eV}{w}
\nc{\eW}{z}
\nc{\subrep}[1]{U_{#1}}
\nc{\catT}[1]{\cT_{#1}}
\nc{\catR}[2]{{_{#1} \cR_{#2}}}
\nc{\CCB}[2]{{_{#1}\ICB_{#2}}}
\nc{\unit}{\mathbf{1}}
\nc{\sfu}{\mathsf{u}}
\nc{\shV}{\sfu}
\nc{\dom}{\ZZ^{\bv}_{\textsf{dom}}}
\nc{\laG}{{\boldsymbol{\lambda}}}
\nc{\muG}{{\boldsymbol{\mu}}}
\nc{\laQ}{\lambda}
\nc{\muQ}{\mu}
\nc{\Ymu}{Y_\mu[\eW_1,\ldots,\eW_N]}

\nc{\bimIK}{\mathcal{P}} 
\nc{\bimKI}{\mathcal{Q}} 
\nc{\Inc}{\mathsf{Inc}}
\nc{\Av}{\mathsf{Av}}

\nc{\objs}{\bt_\RR^{\circ}}
\nc{\paths}{\bt_\RR^{\bullet}}
\nc{\catya}{\sB^\bv}
\nc{\morphism}{\varphi}

\nc{\defparam}{\varkappa}

\title{Quiver gauge theories and symplectic singularities}
\author{Alex Weekes}
\date{}

\begin{document}
\begin{abstract}
Braverman, Finkelberg and Nakajima have recently given a mathematical construction of the Coulomb branches for a large class of $3d$ $\cN =4$ gauge theories, as algebraic varieties with Poisson structure.  They conjecture that these varieties have symplectic singularities.  We confirm this conjecture for all quiver gauge theories without loops or multiple edges, which in particular implies that the corresponding Coulomb branches have finitely many symplectic leaves and have rational Gorenstein singularities. We also give a criterion for proving that any particular Coulomb branch has symplectic singularities, and discuss the possible extension of our results to quivers with loops and/or multiple edges.
\end{abstract}
\maketitle

\section{Introduction}

Braverman, Finkelberg and Nakajima have proposed a mathematical definition for the Coulomb branches of a large class of $3d$ $\cN=4$ gauge theories \cite{Nak6,BFN1}. Their work has attracted considerable interest due to its origin in quantum field theory, and also because Coulomb branches -- along with their duals, Higgs branches -- play a natural role in the mathematical theory of symplectic duality \cite{BLPW, BDGH, Web3, BF2}.  They define Coulomb branches as algebraic varieties with Poisson structures over $\CC$ (in fact, they are defined over $\ZZ$), and conjecture  \cite[\S 3(iv)]{BFN1} that Coulomb branches have {\em symplectic singularities} in the sense of Beauville \cite{Beauville}.  This is a natural conjecture from the point of view of the symplectic duality program, but remains unknown outside of a limited number of cases.

Quiver gauge theories form a particularly interesting family of $3d $ $\cN =4$ gauge theories, whose Higgs branches are Nakajima quiver varieties.  Their Coulomb branches are also important spaces in geometric representation theory: in finite ADE types they are generalized affine Grassmannian slices \cite[Theorem 3.10]{BFN2}, while in affine type A they are Cherkis bow varieties \cite[Theorem 6.18]{NakTak}.  In general, the Coulomb branch for a quiver gauge theory without loops can be thought of as a generalized affine Grassmannian slice for the corresponding symmetric Kac-Moody group \cite{FinkICM}.  These varieties are suggested as a possible setting in which to generalize the geometric Satake correspondence, see \cite{Nak8} and \cite[\S 3(viii)]{BFN2}.

Given a quiver $Q$ along with dimension vectors $\bv, \bw$, we denote the corresponding Coulomb branch by $\cM_C(Q,\bv, \bw)$.   Throughout most of this note, we assume that $Q$ is a {\em simple quiver}: that its underlying graph has no loops or multiple edges (although cycles are still allowed).  
\begin{Theorem}
\label{main thm}
For any simple quiver, $\cM_C(Q,\bv, \bw)$ has symplectic singularities.
\end{Theorem}

This result is already known in certain cases.  In affine type A it follows from the work of Nakajima and Takayama, since bow varieties admit symplectic resolutions \cite[\S 6.2]{NakTak}.  In finite ADE types, recall that $\cM_C(Q, \bv, \bw) \cong \overline{\cW}_{\mu^\ast}^{\lambda^\ast}$ is a generalized affine Grassmannian slice.  These varieties have symplectic singularities by the recent work of Zhou \cite[Theorem 4.7]{Zhou}. Zhou's result subsumes several previous results for finite ADE types: If $\mu^\ast$ is dominant then it is an {ordinary} affine Grassmannian slice, so the above theorem is part of \cite[Theorem 2.7]{KWWY}.  When $\lambda^\ast = 0$ or equivalently $\bw = 0$, then $\overline{\cW}_{\mu^\ast}^0$ is an open Zastava space \cite[Theorem 3.1]{BFN2}, which is smooth \cite[Prop.~3.5]{Finkelberg-Mirkovic} and symplectic \cite[Corollary 1]{FKMM}. Finally, if $\lambda^\ast$ is minuscule and $\mu^\ast \in W \lambda^\ast$, then $\overline{\cW}_{\mu^\ast}^{\lambda^\ast} \cong \mathbb{A}^{2 \langle \rho, \lambda^\ast - \mu^\ast\rangle}$ is smooth and symplectic by \cite[Theorem 2.9]{KryPer},  and in fact its symplectic structure can be identified with the standard one on $\mathbb{A}^{2 \langle \rho, \lambda^\ast - \mu^\ast\rangle} \cong T^\ast \mathbb{A}^{\langle \rho, \lambda^\ast - \mu^\ast\rangle}$. 


We develop a criterion (Theorem \ref{criterion}) for showing that a given Coulomb branch has symplectic singularities.  Applying this criterion in the case of quiver gauge theories,  the proof of Theorem \ref{main thm} reduces to the study of several simple quivers, see Figure \ref{figure}.   
In \S \ref{section: extending} we list those additional cases which would be necessary and sufficient in order to extend the main theorem to all quivers, i.e.~to allow loops and/or multiple edges.  

In addition to providing a definition of Coulomb branches in their foundational paper \cite{BFN1}, Braverman, Finkelberg and Nakajima also developed several techniques which are crucial in the present work.  The most important is a construction of partial resolutions of Coulomb branches using \emph{flavour symmetry}, see \S\ref{section: flavour resolutions}.  This construction can be interpreted as a GIT quotient,  allowing for variation of GIT arguments.   In the case of simple quivers, we show that for generic GIT parameters these partial resolutions have terminal singularities, and we optimistically expect that they are $\QQ$--factorial terminalizations, see Remark \ref{remark: qfactterm}.

In this paper we will not make use of the physical meaning of Coulomb branches, and refer the reader to \cite{BDG, Nak6} for more details and further references.  We note that the symplectic singularities and resolutions of Coulomb branches have been studied in physics, for example by Hanany and collaborators \cite{HananySperling,Hetal,GrimmingerHanany}.

By \cite[Prop.~1.3]{Beauville}, \cite[Theorem 2.3]{Kaledin} and \cite[Prop.~3.7]{Brown-Gordon}, Theorem \ref{main thm} has several immediate consequences:
\begin{Corollary} 
For any simple quiver:
\begin{enumerate}[(a)]
\item $\cM_C(Q,\bv,\bw)$ has rational Gorenstein singularities.

\item $\cM_C(Q,\bv, \bw)$ has finitely many (holomorphic) symplectic leaves.  Moreover, the symplectic leaves are the irreducible components of the smooth loci $X_i^{reg} \subset X_i$ for the stratification
$$
\cM_C(Q, \bv, \bw) = X_0 \supset X_1 \supset X_2 \supset \ldots,
$$
where $X_{i+1}$ is the singular locus of $X_{i}$.

\item The symplectic leaves of $\cM_C(Q,\bv, \bw)$ are smooth connected locally-closed subvarieties.
\end{enumerate}
\end{Corollary}

In finite ADE type $\cM_C(Q, \bv,\bw)$ has an explicit decomposition into symplectic leaves by \cite{MuthiahWeekes}, \cite[Theorem~2.5]{KWWY}, and in affine type A by \cite[\S 4.1]{NakTak}.  These decompositions are very natural from the point of view of affine Grassmannian slices (resp.~bow varieties), and are of representation-theoretic significance.  Some predictions in general are given in \cite[\S 2(ii)]{Nak7},\cite{Hetal,GrimmingerHanany}.

\subsection{Acknowledgements}

This work owes its existence to Braverman, Finkelberg, and Nakajima, and their groundbreaking work on Coulomb branches.
It also would not have been possible without discussions with Gwyn Bellamy, Michael Finkelberg, Dinakar Muthiah, Hiraku Nakajima, Travis Schedler and Oded Yacobi. I am grateful to them for their many helpful suggestions and comments.  Finally, I would like to thank my anonymous referees for their careful reading and feedback.

\section{Recollections on Coulomb branches}

\subsection{The definition}
We briefly overview the construction of Coulomb branches due to Braverman, Finkelberg, and Nakajima, referring the reader to \cite{BFN1} for more details and properties.  

Let $\bG$ be a connected reductive group over $\CC$, and let $\bN$ be a representation of $\bG$ over $\CC$.  It will be convenient to fix a maximal torus $\bT \subset \bG$, with Lie algebra $\ft = \operatorname{Lie} \bT$ and Weyl group $W$.

To the datum $(\bG, \bN)$ there is an associated moduli space $\cR_{\bG, \bN}$ \cite[\S 2(i)]{BFN1}, consisting of triples $(P, \varphi, s)$ where (i) $P$ is principal $\bG$--bundle on the formal disk $D$, (ii) $\varphi : P |_{D^\times} \rightarrow D^\times \times \bG$ a trivialization over the formal punctured disk, and (iii) $s$ is a section of the associated bundle $P \times^\bG \bN$ such that $\varphi(s)$ is regular.  Its equivariant Borel-Moore homology $H_\ast^{\bG_\cO}(\cR_{\bG, \bN})$ has a commutative ring structure by \cite[Theorem 3.10]{BFN1}, and the corresponding affine scheme
\begin{equation*}
\label{eq: CB def}
\cM_C(\bG, \bN) \DEF  \Spec H_\ast^{\bG_\cO}(\cR_{\bG, \bN})
\end{equation*}
is called the {\bf Coulomb branch associated to $(\bG, \bN)$} \cite[Definition 3.13]{BFN1}.  

\begin{Remark}
We will always take Borel-Moore homology with coefficients in $\CC$, and think of $\cM_C(\bG,\bN)$ as a scheme over $\CC$.
\end{Remark}

Some key examples of Coulomb branches are recalled in \S \ref{examples}.  For now, we summarize several important properties:
\begin{Theorem}[\mbox{\cite[Cor.~5.21, 5.22, Prop.~6.12, 6.15]{BFN1}}]
\label{thm: BFN theorem}
$\cM_C(\bG, \bN)$ is an irreducible normal affine variety of dimension $2 \operatorname{rk} \bG$.  It has a Poisson structure which is symplectic on its smooth locus.
\end{Theorem}

The Poisson structure on $\cM_C(\bG,\bN)$ is defined by considering $H_\ast^{\bG_\cO \rtimes \CC^\times}(\cR_{\bG,\bN})$, where the group $\CC^\times$ acts on $\cR_{\bG,\bN}$ via loop rotation of the disc $D$.  This Borel-Moore homology again carries an algebra structure, and is a non-commutative deformation of $H_\ast^{\bG_\cO}(\cR_{\bG,\bN})$, see \cite[\S 3(iv)]{BFN1}. This naturally endows $\cM_C(\bG,\bN)$ with a Poisson structure.

One further property of $\cM_C(\bG,\bN)$ that we will need is described in \cite[\S 3(vii)(a)]{BFN1}: suppose that $\bG = \bG_1 \times \bG_2$ and $\bN = \bN_1\oplus \bN_2$, where $\bN_i$ is a representation of $\bG_i$ alone.  Then
\begin{equation}
\label{eq: product}
\cM_C(\bG, \bN) \cong \cM_C(\bG_1,\bN_1) \times \cM_C(\bG_2, \bN_2)
\end{equation}

\subsection{Flavour resolutions}
\label{section: flavour resolutions}
An important structure for Coulomb branches is the notion of flavour symmetry,	 this lively terminology coming from physics.  For the purposes of this paper we will only consider the following situation: that $\bG$ fits into an exact sequence of connected reductive groups
\begin{equation*}
1 \longrightarrow \bG \longrightarrow \widetilde{\bG} \longrightarrow \bF \longrightarrow 1
\end{equation*}
such that $(i)$ the action of $\bG$ on $\bN$ extends to an action of the larger group $\widetilde{\bG}$, and $(ii)$ $\bF$ is a torus.  In particular, $\bG$ is a normal subgroup of $\widetilde{\bG}$.  We call $\bF$ a {\bf flavour symmetry group}.  We will denote a maximal torus by $\bT \subset \widetilde{\bT} \subset \widetilde{\bG}$, and its Lie algebra by $\tilde{\ft} = \operatorname{Lie}\widetilde{\bT}$.

The flavour symmetry group $\bF$ can be used to construct partial resolutions of the Coulomb branch, as well as to construct associated line bundles, following \cite[\S 3(ix)]{BFN1} and \cite{BFN4}. For these partial resolutions we should consider the moduli space $\cR_{\widetilde{\bG}, \bN}$ for the pair $(\widetilde{\bG}, \bN)$.  It admits a map $\pi : \cR_{\widetilde{\bG}, \bN} \rightarrow \Gr_\bF$ to the affine Grassmannian $\Gr_\bF$ for the group $\bF$. For any coweight $\defparam \in \cochar$ of $\bF$ we define $\cR_{\widetilde{\bG}, \bN} (\defparam) = \pi^{-1}(z^\defparam)$ to be the preimage  of the point $z^\defparam \in \Gr_\bF$.  In particular $\cR_{\widetilde{\bG}, \bN} (0) = \cR_{\bG, \bN}$.  There is a corresponding grading on the equivariant homology ring:
\begin{equation}
\label{eq: grading by coweights}
H_\ast^{\bG_\cO}( \cR_{\widetilde{\bG}, \bN}) = \bigoplus_{\defparam \in \cochar}  H_\ast^{\bG_\cO}\big( \cR_{\widetilde{\bG}, \bN} (\defparam) \big)
\end{equation}
Now, for a fixed coweight $\defparam$ define:
\begin{equation*}
\cM_C^\varkappa(\bG, \bN) \DEF \Proj \Big( \bigoplus_{n \geq 0} H_\ast^{\bG_\cO}\big( \cR_{\widetilde{\bG}, \bN} ( n \defparam) \big) \Big)
\end{equation*}
In particular, $ \cM_C^0(\bG, \bN) = \cM_C(\bG, \bN)$.  Note that we omit $\widetilde{\bG}$ from our notation, though it is important in the definition of $\cM_C^\defparam(\bG,\bN)$. There is a natural projective morphism 
\begin{equation}
\label{eq: partial reso}
\pi^\defparam : \cM_C^\defparam(\bG, \bN)  \rightarrow \cM_C(\bG, \bN),
\end{equation}  
which is  birational by \cite[Remark 1.1]{BFN4}.  In good cases this map will be a symplectic resolution, see e.g.~\S \ref{example 1}, \ref{example 2}.

\begin{Remark}
\label{rem: hamiltonian reduction}
As explained in \cite[\S 3(ix)]{BFN1}, this construction can be naturally identified with a Hamiltonian reduction:
$$
\cM^\defparam_C(\bG, \bN) \cong \cM_C(\widetilde{\bG}, \bN)\ssslash_\defparam \bF^\vee
$$ 
Here $\bF^\vee = \Spec \CC[\cochar]$ is the dual torus to $\bF$, and $\defparam$ is thought of as a GIT parameter for $\bF^\vee$.  In particular, this perspective shows that $\cM_C^\defparam(\bG,\bN)$ is finite type, and that the map $\pi^\defparam$ is Poisson.
\end{Remark}

\begin{Remark}
The $\cM_C^\defparam(\bG,\bN)$ are irreducible varieties by Lemma \ref{cor: integral} below. We expect that they are also normal, and symplectic on their smooth loci. In Theorem \ref{thm: BFN theorem for resolutions},  we will prove that these properties hold in the case of quiver gauge theories for simple quivers.
\end{Remark}

%

We now consider a flavoured generalization of (\ref{eq: product}).  Following the same notation, suppose that the theories $(\bG_i, \bN_i)$ both admit the same flavour symmetry group $\bF$.  Then so does $(\bG, \bN)$: 
\begin{equation}
\begin{tikzcd}
1 \ar[d,equal] \ar[r] & \bG_1 \times \bG_2 \ar[d,equal] \ar[r] & \widetilde{\bG_1} \times \widetilde{\bG_2} \ar[r] & \bF \times \bF \ar[r] & 1 \ar[d, equal] \\
1 \ar[r] & \bG \ar[r] & \widetilde{\bG} \ar[u, hookrightarrow] \ar[r] & \bF \ar[u, hookrightarrow, "\Delta"] \ar[r] & 1 
\end{tikzcd}
\end{equation}
That is, we define $\widetilde{\bG} \subset \widetilde{\bG_1}\times \widetilde{\bG_2}$ as the preimage of the diagonal $\bF \subset \bF \times \bF$.  Since $\widetilde{\bG}_1\times \widetilde{\bG}_2$ naturally acts on $\bN = \bN_1 \oplus \bN_2$, so does its subgroup $\widetilde{\bG}$.
\begin{Lemma}
\label{lemma: segre}
With the above choice of $\widetilde{\bG}$, for any $\defparam \in \cochar$ we have 
$$\cM_C^\defparam(\bG, \bN) \cong \cM_C^\defparam(\bG_1, \bN_1) \times \cM_C^\defparam(\bG_2, \bN_2)$$
\end{Lemma}
\begin{proof}
This is a Segre product: by \cite[Exercise II.5.11]{Hartshorne}, it suffices to show that for $n\geq 0$
$$
H_\ast^{\bG_\cO}\big( \cR_{\widetilde{\bG}, \bN} ( n \defparam) \big) \cong H_\ast^{\bG_{1,\cO}}\big( \cR_{\widetilde{\bG_1}, \bN_1} ( n \defparam) \big) \otimes_\CC H_\ast^{\bG_{2,\cO}}\big( \cR_{\widetilde{\bG_2}, \bN_2} ( n \defparam) \big)
$$
By K\"unneth's theorem, this follows from $\cR_{\widetilde{\bG}, \bN}(n\defparam) \cong \cR_{\widetilde{\bG_1}, \bN_1}(n \defparam) \times  \cR_{\widetilde{\bG_2}, \bN_2}(n \defparam)$, which in turn follows from the definitions.
\end{proof}


\subsection{The integrable system and hyperplane arrangement} 
We will freely identify equivariant cohomology rings with the coordinate rings of quotients of Cartan subalgebras, for example:
$$
H_\bT^\ast(pt) \cong \CC[\ft] \ \text{ and } \ H_\bG^\ast(pt) \cong \CC[\ft]^W
$$ 
There is an inclusion of rings $\CC[\ft]^W  \hookrightarrow H_\ast^\bG( \cR_{\bG, \bN})$ by \cite[\S 3(vi)]{BFN1}, which induces a morphism
\begin{equation*}
\varpi: \cM_C(\bG, \bN) \longrightarrow  \ft \sslash W,
\end{equation*}
Composing with $\pi^\defparam$ we obtain a map $\varpi^\defparam: \cM^\defparam_C(\bG, \bN) \rightarrow \ft \sslash W$.   The maps $\varpi, \varpi^\defparam$ are integrable systems for the Poisson structure, in that $\CC[\ft]^W$ is a maximal Poisson-commutative subalgebra by \cite[Cor.~5.21]{BFN1}.

\begin{Lemma}
\label{lemma: faithfully flat}
The morphisms $\varpi, \varpi^\defparam$ are faithfully flat.
\end{Lemma}
\begin{proof}
\cite[Lemma 2.6]{BFN1} shows that $H_\ast^{\bG_\cO}(\cR_{{\bG}, \bN})$ is free over $\CC[\ft]^W$. More precisely, there is a non-canonical isomorphism of $\CC[\ft]^W$--modules (c.f.~also \cite[\S 6(i)]{BFN1}, \cite[Lemma 6.2]{BFM})
\begin{equation}
\label{eq: decomp into lambdas}
H_\ast^{\bG_\cO}(\cR_{{\bG}, \bN}) \cong \bigoplus_\lambda H_\ast^{\bG_\cO}(\cR_{{\bG}, \bN}^\lambda),
\end{equation}
where $\lambda$ runs over dominant coweights of $\bG$.  Here $\cR_{{\bG}, \bN}^\lambda$ denotes the preimage in $\cR_{\bG,\bN}$ of the $\bG_\cO$-orbit $\Gr^\lambda_\bG \subset \Gr_\bG$.    Up to a grading shift, there is an isomorphism $H_\ast^{\bG_\cO}(\cR_{{\bG}, \bN}^\lambda) \cong H_\ast^{\bG}( \bG / \bP_\lambda)$, where $\bG / \bP_\lambda$ is a partial flag variety corresponding to $\lambda$. Each $\CC[\ft]^W$-module $H_\ast^{\bG}( \bG / \bP_\lambda)$ is free, which proves that $H_\ast^{\bG_\cO}(\cR_{{\bG}, \bN})$ is free. In particular, the map $\varpi$ is faithfully flat.

For any $n$, the $\CC[\ft]^W$-module $H_\ast^{\bG_\cO}(\cR_{\widetilde{\bG}, \bN} (n \defparam))$ has a decomposition analogous to (\ref{eq: decomp into lambdas}), but where the sum runs over coweights $\tilde{\lambda}$ of $\widetilde{\bG}$ which are lifts of $n \defparam$. By similar reasoning as above, this proves that $H_\ast^{\bG_\cO}(\cR_{\widetilde{\bG}, \bN} (n \defparam))$ is free over $\CC[\ft]^W$. Consequently the morphism $\varpi^\defparam$ is flat\cite[\href{https://stacks.math.columbia.edu/tag/0D4C}{Tag 0D4C}]{stacks-project}.  But $\varpi^\defparam$ is also surjective: we know already that this is true of $\varpi$, while $\pi^\defparam$ is surjective since it is proper and dominant.
\end{proof}

\subsubsection{Generalized root filtration}
Following \cite[Definition 5.2]{BFN1}, a {\bf generalized root for $(\bG, \bN)$} is an element of $\ft^\ast$ which is either (I) a non-zero weight of $\bN$ or (II) a root of $\operatorname{Lie} \bG$.  Let $\varphi_1,\ldots, \varphi_r$ be the distinct generalized roots for $(\bG, \bN)$.  For any $t \in \ft$, define
\begin{equation}
\label{eq: codim function}
\codim_{\bG, \bN}(t) = \codim_\ft \Big( \bigcap_{\varphi_i(t) = 0} \mathbb{V}(\varphi_i) \Big)
\end{equation}
Here $\mathbb{V}(f) \subset \ft$ denotes the vanishing locus of a function $f$ on $\ft$. Note that the intersection on the right-hand side above is simply the intersection of all generalized root hyperplanes containing $t$.  For any $k\geq 0$, we define a subvariety of $\ft$ by
\begin{equation}
\label{eq: codim subsets}
\ft^{(k)} = \left\{ t\in \ft \ : \ \codim_{\bG, \bN}(t) \geq k \right\}
\end{equation}
It is not hard  to see that $\ft^{(k)}$ is a Zariski closed subset: it is the union of all codimension $\geq k$ intersections of generalized root hyperplanes.  In particular $\ft^{(k)}$ may be empty, but if $\ft^{(k)}\neq \emptyset$ then $\codim_\ft \ft^{(k)} = k$. 

\begin{Remark}
Up to a minor discrepancy when $\dim \ft = 2$, in the notation of \cite[\S 5(vi)]{BFN1} 
$$\ft^{\circ} = \ft \setminus \ft^{(1)}, \qquad \ft^\bullet = \ft \setminus \ft^{(2)}$$
\end{Remark}

\subsubsection{Preimages of local rings}

Let $t \in \ft$.  Let $Z_\bG(t) \subset \bG$ be the centralizer of $t$, and $\bN^t \subset \bN$ the subspace of $t$--invariants.  Then $\bT$ is a maximal torus of $Z_\bG(t)$, and the Weyl group of $Z_\bG(t)$ is  the stabilizer $W_t \subset W$ of $t$.  Note that $H_{Z_\bG(t)}^\ast(pt) \cong \CC[\ft]^{W_t}$.  Also note that
\begin{equation}
\label{eq: codim = codim}
\codim_{\bG, \bN}(t) = \codim_{Z_\bG(t), \bN^t}( 0 )
\end{equation}

As explained in \cite[\S 6(v)]{BFN1}, there is an isomorphism of rings
\begin{equation}
\label{eq: fibers BFN iso}
H_\ast^{\bG_\cO} (\cR_{\bG, \bN} ) \otimes_{\CC[\ft]^W} \CC[\ft]^{W_t}_t \ \stackrel{\sim}{\longrightarrow} \  H_\ast^{Z_\bG(t)_\cO} (\cR_{Z_\bG(t), \bN^t} ) \otimes_{\CC[\ft]^{W_t}} \CC[\ft]^{W_t}_t  
\end{equation}
Here $\CC[\ft]^{W_t}_t$ denotes the localization of $\CC[\ft]^{W_t}$ at the maximal ideal corresponding to the image of $t$ under $\ft \rightarrow \ft \sslash W_t$.  Similarly, we'll write $\CC[\ft]_t$ for localization of $\CC[\ft]$ at $t$, and $\CC[\ft]^W_t$ for the localization of $\CC[\ft]^W$ at the image of $t$ under $\ft \rightarrow \ft \sslash W$.  The group $W_t$ acts on $\CC[\ft]_t$ with invariants $(\CC[\ft]_t)^{W_t} = \CC[\ft]_t^{W}$, as follows from the fact that $\operatorname{Frac}(\CC[\ft])^{W_t} = \operatorname{Frac}(\CC[\ft]^{W_t})$. 

Since $\CC[\ft]^W_t \hookrightarrow \CC[\ft]^{W_t}_t$, by composition we obtain from (\ref{eq: fibers BFN iso}) a map of rings
\begin{equation*}
H_\ast^{\bG_\cO} (\cR_{\bG, \bN} ) \otimes_{\CC[\ft]^W} \CC[\ft]^{W}_t \hooklongrightarrow  H_\ast^{Z_\bG(t)_\cO} (\cR_{Z_\bG(t), \bN^t} ) \otimes_{\CC[\ft]^{W_t}} \CC[\ft]^{W_t}_t
\end{equation*}
Denote the corresponding affine schemes by $\cM_C(\bG, \bN)_t$ and $\cM_C(Z_\bG(t), \bN^t)_t$, respectively.  Then we can rephrase (\ref{eq: fibers BFN iso}) as a Cartesian diagram of schemes:
\begin{equation}
\label{eq: fibers map 2}
\begin{tikzcd}
\cM_C(Z_\bG(t), \bN^t)_t \ar[r] \ar[d] & \cM_C(\bG, \bN)_t \ar[d] \\
\Spec \CC[\ft]^{W_t}_t \ar[r] & \Spec \CC[\ft]^{W}_t 
\end{tikzcd}
\end{equation}

This picture extends to include flavour resolutions. In the notation of Section \ref{section: flavour resolutions},  first observe that there is a corresponding exact sequence of centralizer subgroups
\begin{equation*}
1 \longrightarrow Z_\bG(t) \longrightarrow Z_{\widetilde{\bG}}(t) \longrightarrow \bF \longrightarrow 1
\end{equation*}
In particular, for any $\defparam\in \cochar$ there is a corresponding variety $\cM_C^\defparam(Z_\bG(t), \bN^t)$.  Note that $
\bigoplus_{n\geq 0}H_\ast^{\bG_\cO} \big(\cR_{\widetilde{\bG}, \bN} ( n \defparam) \big)
$
is a graded algebra over $\CC[\ft]^W$, which lies in degree zero.  Thus $\Proj$ is compatible with base change to any trivially graded $\CC[\ft]^W$--algebra, see \cite[\href{https://stacks.math.columbia.edu/tag/01N2}{Tag 01N2}]{stacks-project}, and in particular:
\begin{equation}
\label{eq: base change}
\cM^\defparam_C(\bG, \bN) \times_{\ft \sslash W} \Spec \CC[\ft]^W_t \ \cong \ \Proj \Big( \bigoplus_{n \geq 0} H_\ast^{\bG_\cO}\big( \cR_{\widetilde{\bG}, \bN} ( n \defparam)\big) \otimes_{\CC[\ft]^W} \CC[\ft]^W_t \Big)
\end{equation}
We denote this scheme by $\cM_C^\defparam(\bG, \bN)_t$.  Similarly, we write $\cM_C^\defparam(Z_{\bG}(t), \bN^t)_t$ for an analogous base change along $\CC[\ft]^{W_t} \rightarrow \CC[\ft]^{W_t}_t$.  

The following result directly generalizes (\ref{eq: fibers map 2}):
\begin{Theorem}
\label{prop: fibers via centralizers and fixed}
For any $\defparam\in \cochar$ and $t\in \ft$, there is a Cartesian diagram of schemes:
$$
\begin{tikzcd}
\cM_C^\defparam(Z_\bG(t), \bN^t)_t \ar[r] \ar[d]  & \cM_C^\defparam(\bG, \bN)_t \ar[d] \\
\Spec \CC[\ft]^{W_t}_t \ar[r] & \Spec \CC[\ft]^{W}_t
\end{tikzcd}
$$
Moreover, the horizontal arrows are both \'etale and surjective.
\end{Theorem}

\begin{proof}
By a variation on the far right column of the diagram \cite[Equation (5.24)]{BFN1}, there is a diagram of $\CC[\ft]$--algebras
$$
\begin{tikzcd}
H_\ast^{\bT_\cO}(\cR_{Z_{\widetilde{\bG}}(t), \bN^t}) & H_\ast^{\bT_\cO}(\cR_{Z_{\widetilde{\bG}}(t), \bN}) \ar[l,"\bz^\ast"'] \ar[r, "\iota_\ast"] & H_\ast^{\bT_\cO}(\cR_{\widetilde{\bG}, \bN}) 
\end{tikzcd}
$$
Both maps respect gradings by coweights $\cochar$, defined as in (\ref{eq: grading by coweights}), since they are induced by maps of spaces over $\Gr_\bF$. Both maps also respect natural $W_t$--actions, by \cite[Lemma 3.19]{BFN1}.  Finally, by the localization theorem both maps become isomorphisms over $\CC[\ft]_t$.  Therefore, the map $ \bz^\ast \iota_\ast^{-1}$  induces a $W_t$--equivariant isomorphism of $\cochar$-graded algebras
\begin{equation*}
 H_\ast^{\bT_\cO} \big( \cR_{\widetilde{\bG}, \bN} \big) \otimes_{\CC[\ft]} \CC[\ft]_t  \ \stackrel{\sim}{\longrightarrow} \ H_\ast^{\bT_\cO} \big(\cR_{Z_{\widetilde{\bG} }(t), \bN^t} \big)\otimes_{\CC[\ft]} \CC[\ft]_t 
\end{equation*}

Taking $W_t$-invariants on both sides, we claim that we obtain an isomorphism 
\begin{equation}
\label{eq: localization isomorphism with flavour 2} 
 H_\ast^{\bG_\cO} \big( \cR_{\widetilde{\bG}, \bN} \big) \otimes_{\CC[\ft]^W} \CC[\ft]^{W_t}_t  \ \stackrel{\sim}{\longrightarrow} \  H_\ast^{{Z_{\bG}(t)}_\cO} \big(\cR_{Z_{\widetilde{\bG} }(t), \bN^t} \big)\otimes_{\CC[\ft]^{W_t}} \CC[\ft]^{W_t}_t 
\end{equation}
To prove this claim, we identify the left side of (\ref{eq: localization isomorphism with flavour 2}) with the $W_t$--invariants of the left side of the preceding equation; a similar argument applies to the right side.   We have $H_\ast^{\bT_\cO} (\cR_{\widetilde{\bG}, \bN} ) \cong H_\ast^{\bG_\cO} (\cR_{\widetilde{\bG}, \bN} ) \otimes_{\CC[\ft]^W} \CC[\ft]$  by \cite[Lemma 5.3]{BFN1}, and therefore
\begin{equation}
\label{eq: base changing}
\phantom{\cong} H_\ast^{\bT_\cO} \big( \cR_{\widetilde{\bG}, \bN} \big) \otimes_{\CC[\ft]} \CC[\ft]_t  \cong H_\ast^{\bG_\cO} \big( \cR_{\widetilde{\bG}, \bN} \big) \otimes_{\CC[\ft]^W}\CC[\ft] \otimes_{\CC[\ft]} \CC[\ft]_t \cong H_\ast^{\bG_\cO} \big( \cR_{\widetilde{\bG}, \bN} \big) \otimes_{\CC[\ft]^W}  \CC[\ft]_t 
\end{equation}
By the proof of Lemma \ref{lemma: faithfully flat}, $H_\ast^{\bG_\cO}(\cR_{\widetilde{\bG},\bN})$ is free over $\CC[\ft]^W$. Both carry trivial actions of $W_t$, while $(\CC[\ft]_t)^{W_t} = \CC[\ft]_t^{W_t}$.  Taking $W_t$--invariants on the right side of (\ref{eq: base changing}), we get the left side of (\ref{eq: localization isomorphism with flavour 2}).

%

The isomorphism (\ref{eq: localization isomorphism with flavour 2}) again respects $\cochar$-gradings.  Let us restrict it to the direct sum of all graded pieces of degrees $n \defparam \in \cochar$, where $n\geq 0$, and then take $\Proj$ of both sides.  On the right side we obtain $\cM_C^\defparam(Z_\bG(t), \bN^t)_t$. On the left side we obtain the base change of $\cM_C^\defparam(\bG,\bN)_t$ along $\Spec \CC[\ft]^{W_t}_t \rightarrow \Spec \CC[\ft]^{W}_t$, because  of  (\ref{eq: base change}).  This establishes the desired Cartesian diagram.


Finally, \'etale morphisms are stable under base change \cite[\href{https://stacks.math.columbia.edu/tag/02GO}{Tag 02GO}]{stacks-project}, as are surjective morphisms \cite[\href{https://stacks.math.columbia.edu/tag/01S1}{Tag 01S1}]{stacks-project}, so it suffices to prove that the bottom horizontal arrow above is \'etale and surjective.  To this end we observe that the natural map $\ft \sslash W_t \rightarrow \ft \sslash W$ is \'etale at $t$, and so the map of local rings $\CC[\ft]^{W}_t \hookrightarrow \CC[\ft]^{W_t}_t$ is \'etale. It is also surjective, so we are done.
\end{proof}

Since \'etale maps preserve smoothness and normality, we immediately obtain the following:
\begin{Corollary}
\label{cor: etale smooth}
$\cM_C^\defparam (\bG, \bN)_t$ is smooth (resp.~normal) iff $\cM_C^\defparam (Z_\bG(t),\bN^t)_t$ is smooth (resp.~normal).
\end{Corollary}

The benefit of this result is that the theory $(Z_\bG(t), \bN^t)$ is often far simpler than the original theory $(\bG, \bN)$, and in particular the corresponding Coulomb branch may be better understood.

Next, we give a generalization of \cite[Corollary 5.12]{BFN1}:
\begin{Lemma}
\label{cor: integral}
$\cM_C^\defparam(\bG,\bN)$ is an integral scheme.
\end{Lemma}
\begin{proof}
It suffices to prove that $H_\ast^{\bG_\cO}(\cR_{\widetilde{\bG},\bN})$ is a domain, since $\cM_C^\defparam(\bG,\bN)$ is defined as the $\Proj$ of its subalgebra. We will use a variation on the top row of \cite[Equation (5.24)]{BFN1}:
$$
\begin{tikzcd}
\CC[\ft \times \widetilde{\bT}^\vee] = H_\ast^{\bT_\cO}(\Gr_{\widetilde{\bT}}) & H_\ast^{\bT_\cO}(\cR_{\widetilde{\bT},\bN}) \ar[l,"\bz^\ast"'] \ar[r,"\iota_\ast"] & H_\ast^{\bT_\cO}(\bR_{\widetilde{\bG},\bN})
\end{tikzcd}
$$
By the localization theorem, the resulting morphisms become isomorphisms over $\CC(\ft)$. Since $H_\ast^{\bT_\cO}(\bR_{\widetilde{\bG},\bN})$ is free over $\CC[\ft]$, it follows that there is an injection
$$
\bz^\ast (\iota_\ast)^{-1}: H_\ast^{\bT_\cO}(\bR_{\widetilde{\bG},\bN}) \hooklongrightarrow \CC[\ft \times \widetilde{\bT}^\vee]\otimes_{\CC[\ft]} \CC(\ft)
$$
Next, as in \cite[Lemma 5.3]{BFN1} we have 
$$
H_\ast^{\bG_\cO}(\cR_{\widetilde{\bG},\bN}) \cong H_\ast^{\bT_\cO}(\cR_{\widetilde{\bG},\bN})^W \hooklongrightarrow H_\ast^{\bT_\cO}(\cR_{\widetilde{\bG},\bN}),
$$
Altogether, we see that $H_\ast^{\bG_\cO} \big( \cR_{\widetilde{\bG}, \bN} \big)$ embeds into the domain $\CC[\ft\times \widetilde{\bT}^\vee] \otimes_{\CC[\ft]}\CC(\ft)$.
\end{proof}

\begin{Proposition}
\label{Prop: Poisson maps}
The morphism $\cM_C^\defparam(Z_\bG(t),\bN^t)_t \rightarrow \cM_C^\defparam(\bG,\bN)_t$ is Poisson.
\end{Proposition}
\begin{proof}
The following commutative square relates those from (\ref{eq: fibers map 2}) and Theorem \ref{prop: fibers via centralizers and fixed}:
$$
\begin{tikzcd}
\cM_C^\defparam(Z_\bG(t), \bN^t)_t \ar[r] \ar[d]  & \cM_C^\defparam(\bG, \bN)_t \ar[d] \\ 
\cM_C(Z_\bG(t), \bN^t)_t \ar[r]& \cM_C(\bG, \bN)_t  
\end{tikzcd}
$$
The vertical arrows are Poisson by Remark \ref{rem: hamiltonian reduction}, and birational by \cite[Remark 1.1]{BFN4}. Moreover all the objects above are integral schemes.  Therefore, by a similar argument to \cite[Lemma 4.12]{KPW}, to prove that the top arrow is Poisson, it suffices to prove that the bottom arrow is Poisson.

Because of the commutativity of the diagram \cite[Equation (5.24)]{BFN1}, there is a triangle
$$
\begin{tikzcd}[sep = tiny]
 H_\ast^{\bG_\cO} \big( \cR_{\bG, \bN} \big) \otimes_{\CC[\ft]^W} \CC[\ft]^{W}_t \ar[dr,hook] \ar[rr,hook] & & H_\ast^{{Z_{\bG}(t)}_\cO} \big(\cR_{Z_{{\bG}}(t), \bN^t} \big)\otimes_{\CC[\ft]^{W_t}} \CC[\ft]^{W_t}_t  \ar[dl,hook']\\
 & \CC[\ft \times {\bT}]\otimes_{\CC[\ft]} \CC(\ft) &
\end{tikzcd}
$$
The horizontal arrow comes from (\ref{eq: fibers BFN iso}), while the downward arrows are both defined as in the proof of Proposition \ref{cor: integral}.  Both downward embeddings can be quantized, see \cite[Remark 5.23]{BFN1}, and in particular they are Poisson. This implies that the horizontal arrow is also Poisson, as claimed.
\end{proof}

%
%
%
%
%
%

\subsection{Symplectic singularities}
\label{section: SS}

Let $X$ be a normal Poisson algebraic variety over $\CC$, such that the restriction of its Poisson structure to the smooth locus $X^{reg}$ is non-degenerate.  Denote the corresponding symplectic form on $X^{reg}$ by $\omega^{reg}$.  Following \cite[Def.~1.1]{Beauville}, we say that $X$ has {\bf symplectic singularities} if for some (equivalently, any) resolution of singularities $\pi : Y \rightarrow X$, the form $\pi^\ast(\omega^{reg})$ extends to a regular 2-form on all of $Y$.  See \cite{Fu,Bellamy} for excellent overviews of these varieties and related topics.

The next result is inspired by \cite[Prop.~6.12, 6.15]{BFN1} as well as \cite[Prop.~2.8]{BFM}.
\begin{Theorem}
\label{criterion: normal symplectic}
Fix $\defparam \in \cochar$.  Suppose that $\cM_C^\defparam(Z_\bG(t), \bN^t)_t$ is normal, and that its Poisson structure is symplectic on its smooth locus, for all $t\in \ft \setminus \ft^{(2)}$.  Then $\cM_C^\defparam(\bG,\bN)$ is normal, and its Poisson structure is symplectic on its smooth locus.
\end{Theorem}
\begin{proof}
Consider the open subscheme $U = (\varpi^{\defparam})^{-1}\big((\ft \setminus \ft^{(2)})\sslash W\big)$.  Under our assumptions, Corollary \ref{cor: etale smooth} and Proposition \ref{Prop: Poisson maps} imply that $U$ is normal, and symplectic on its smooth locus. If $\ft^{(2)} = \emptyset$, then $U = \cM_C^\defparam(\bG,\bN)$ and we are done.  Otherwise $\ft^{(2)} \neq \emptyset$, so $\operatorname{codim} \ft^{(2)} = 2$.  The map $\varpi^\defparam$ is faithfully flat by Lemma \ref{lemma: faithfully flat}, so the complement of $U$ has codimension $2$ by \cite[Cor~6.1.4]{EGAIV2}. 
Looking at smooth loci, it follows that $U^{reg}$ has complement of codimension $\geq 2$ inside   $\cM_C^\defparam(\bG,\bN)^{reg}$. The degeneracy locus of the Poisson bivector on $\cM_C^\defparam(\bG,\bN)^{reg}$ is a divisor, if non-empty. But $U^{reg}$ is symplectic, so the degeneracy locus must be trivial.  Therefore $\cM_C^\defparam(\bG,\bN)^{reg}$ is symplectic.

Next, we show that every open affine $V \subset   \cM_C^\defparam(\bG,\bN)$ is normal. Note that $U\cap V$ is normal, and the complement of $U\cap V$ in $V$ has codimension $\geq 2$. We claim that every regular function on $U\cap V$ extends to $V$.  Assuming this claim, \cite[Lemma 6.13]{BFN1} implies that $V$ is normal.

To prove the claim, consider the restricted map $\varpi^\defparam|_V: V\rightarrow \ft\sslash W$. It is flat, since $V\subset \cM_C^\defparam(\bG,\bN)$ is open and the morphism $\varpi^\defparam$ is flat by Lemma \ref{lemma: faithfully flat}.  Thus $\cF = (\varpi^\defparam|_V)_\ast \cO_V$ is a flat quasi-coherent sheaf on $\ft\sslash W$. Note that $\ft\sslash W$ is normal, and $\ft^{(2)}\sslash W$ has codimension 2.  By a version of the algebraic Hartogs's theorem, it follows that under the inclusion $j: (\ft\setminus \ft^{(2)})\sslash W \hookrightarrow \ft \sslash W$, the natural map $\cF \stackrel{\sim}{\longrightarrow} j_\ast j^\ast \cF$  is an isomorphism. This is precisely the claimed extension property.
\end{proof}

Finally, we obtain a criterion for proving that a Coulomb branch has symplectic singularities:
\begin{Theorem}
\label{criterion}
Fix $\defparam \in \cochar$.  Suppose that $\cM^\defparam_{C}( Z_\bG(t), \bN^t)_t$ is smooth and symplectic  for all $t\in \ft \setminus \ft^{(4)}$. Then $\cM_C^\defparam(\bG, \bN)$ and $\cM_C(\bG, \bN)$ have symplectic singularities.  In this case $\cM^\defparam_C(\bG, \bN)$ has terminal singularities, and its singular locus has codimension $\geq 4$.
\end{Theorem}
\begin{proof}
Under these assumptions Theorem \ref{criterion: normal symplectic} implies that $\cM_C^\defparam(\bG,\bN)$ is normal, and symplectic on its smooth locus.  Moreover, Corollary \ref{cor: etale smooth} and Proposition \ref{Prop: Poisson maps} imply that the preimage 
$$
U = (\varpi^\defparam)^{-1}\big((\ft \setminus \ft^{(4)})\sslash W\big)\subset \cM^\defparam_C(\bG, \bN)
$$ 
is smooth and symplectic.  If $\ft^{(4)} = \emptyset$ then $U = \cM_C^\defparam(\bG,\bN)$ is itself smooth and symplectic.

Otherwise $\operatorname{codim} \ft^{(4)} = 4$. As in the proof of the previous theorem, we conclude that the complement of $U$ has codimension $4$.  Therefore the singular locus of $\cM^\defparam_C(\bG, \bN)$ has codimension $\geq 4$. By a result of Flenner \cite{Flenner}, it follows that for any resolution of singularities $\pi:Y\rightarrow \cM_C^\defparam(\bG,\bN)$ and any 2-form $\omega$ on the smooth locus $\cM_C^\defparam(\bG,\bN)^{reg}$, the pull-back $\pi^\ast \omega$ extends to a regular 2-form on $Y$.  In particular this applies to $\omega = \omega^{reg}$ the symplectic form on $\cM_C^\defparam(\bG,\bN)^{reg}$, which shows that $\cM_C^\defparam(\bG,\bN)$ has symplectic singularities.

Since $\cM_C^\defparam(\bG,\bN)$ has symplectic singularities and its singular locus has codimension $\geq 4$,  it has terminal singularities by \cite[Cor.~1]{Namikawa}.  

In either case,  since $\pi^\defparam: \cM^\defparam_C(\bG,\bN)\rightarrow \cM_C(\bG,\bN)$ is a proper birational Poisson map, we conclude that $\cM_C(\bG, \bN)$ has symplectic singularities by \cite[Lemma 6.12]{BS}.
\end{proof}


\section{Quiver gauge theories}
\subsection{Definition}
\label{sec: QGT}
Let $Q = (I, E)$ be a quiver with vertices $I$ and arrows $E$.  For $e\in E$ we denote by $s(e), t(e) \in I$ its source and target .  Let $\bv, \bw \in \ZZ_{\geq 0}^I$ be two dimension vectors, to which we associate the vector spaces $V_i = \CC^{\bv_i}, W_i = \CC^{\bw_i}$ for each $i\in I$.  To the datum $(Q, \bv, \bw)$ we associate the pair
\begin{equation}
\label{groups from quiver data}
\bG = \prod_{i\in I} \GL(\bv_i), \quad \bN = \bigoplus_{e\in E} \Hom( V_{s(e)}, V_{t(e)} ) \oplus \bigoplus_{i \in I} \Hom( W_i, V_i)
\end{equation}
The group $\bG$ naturally acts on $\bN$, and we call $(\bG, \bN)$ a {\bf quiver gauge theory datum of type $Q$}, even if some $\bv_i$ or $\bw_i$ are zero.  We will always assume that $Q$ is {\bf simple}: that its underlying graph has no loops or multiple edges.  We denote the corresponding Coulomb branch by
\begin{equation*}
\cM_C(Q, \bv, \bw) \DEF \cM_C(\bG, \bN)
\end{equation*}

\begin{Remark}
\label{remark: no summand}
If for some $i\in I$ we have $\bw_i \neq 0$ but $\bv_i = 0$, then the corresponding summand $\Hom(W_i,V_i) = 0$ in (\ref{groups from quiver data}). Note that in this case we obtain the same pair $(\bG, \bN)$ by setting $\bw_i = 0$.
\end{Remark}

\begin{Remark}
Up to isomorphism, $\cM_C(Q,\bv,\bw)$ does not depend on the orientation of the quiver $Q$, see \cite[\S 6(viii)]{BFN1}.
\end{Remark}

\begin{Remark}
It is convenient to encode the datum $(Q, \bv, \bw)$ diagrammatically.  For example, if we let $Q$ be the oriented $A_3$ quiver $1\rightarrow 2 \leftarrow 3$ with dimension vectors $\bv = (3,1,2)$,  $\bw = (4,0,1)$, we get the diagram
$$
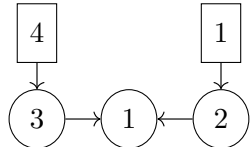
\begin{tikzcd}[sep=small,cells={nodes={draw=black, ellipse,anchor=center,minimum height=2em}}]
|[draw=black, rectangle]| 4 \ar[d] & & |[draw=black, rectangle]| 1 \ar[d] \\
3 \ar[r] & 1 & 2 \ar[l]
\end{tikzcd}
$$
Circled nodes encode the dimensions $\bv_i$, and boxed nodes the $\bw_i$.  We often omit nodes where $\bv_i=0$ or $\bw_i = 0$. 
In particular, an isolated node $\begin{tikzcd}[sep=small,cells={nodes={draw=black, ellipse,anchor=center,minimum height=2em}}]
n
\end{tikzcd} $ corresponds to the ``pure'' theory $(\bG, \bN) = (\GL(n), 0)$.
\end{Remark}

\begin{Remark}
The corresponding Higgs branch $\cM_H(Q, \bv, \bw)$ is a Nakajima quiver variety, as first defined in \cite{Nak3}: it is the Hamiltonian reduction $T^\ast \bN \ssslash \bG$.  These varieties have symplectic singularities \cite[Theorem 1.2]{BS}, and often have symplectic resolutions.
\end{Remark}

\subsubsection{Flavour symmetry}
Given $(Q, \bv, \bw)$ as above, we consider the flavour symmetry group 
\begin{equation}
\label{eq: flavour symmetry group}
\bF  = \bF(Q, \bv, \bw) \DEF \prod_{\substack{e \in E, \\ \bv_{s(e)}\bv_{t(e)} \neq 0}} \CC^\times \times \prod_{\substack{i\in I,\\ \bw_i \bv_i \neq 0}} ( \CC^\times)^{\bw_i}, 
\end{equation}
We think of this group as consisting of a dilating $\CC^\times$ factor for each arrow $e \in E$, along with the diagonal maximal torus $ (\CC^\times)^{\bw_i} \subset \GL(\bw_i)$ for each $i\in I$\footnote{We have elected to drop all factors where  $\Hom(V_{s(e)}, V_{t(e)})= 0$ or $\Hom(W_i, V_i) = 0$, respectively.}.  Thus the group $\bF$ acts on $\bN$ as follows: for $e\in E$ the factor of $\CC^\times$ scales $\Hom(V_{s(e)}, V_{t(e)})$ with weight 1, while for $i\in I$ the torus $(\CC^\times)^{\bw_i}$ acts on $W_i = \CC^{\bw_i}$ in the natural way.   

The actions of $\bG$ and $\bF$ commute, and we set 
\begin{equation*}
\widetilde{\bG} = \bG \times \bF
\end{equation*}
with its natural action on $\bN$.  

The center $Z(\bG) \cong \prod_{i \in I} \CC^\times$ consists of tuples of scalar matrices, and an element $z = (z_i)_{i\in I} \in \prod_{i \in I} \CC^\times$ acts on $\Hom(V_{s(e)}, V_{t(e)})$ by $z_{t(e)} z_{s(e)}^{-1}$ and on $\Hom(W_i, V_i)$ by $z_i$.  This is equal to the action of the element $\big( (z_{t(e)} z_{s(e)}^{-1}),  (z_i^{-1},\ldots,z_i^{-1}) \big) \in \bF$, and this correspondence defines a homomorphism $\phi: Z(\bG) \rightarrow \bF$ such that the action of $Z(\bG)$ on $\bN$ factors through $\phi$.   Via this homomorphism $\phi$, cocharacters of $Z(\bG)$ give cocharacters of $\bF$.  

\begin{Lemma}
Let $\defparam, \defparam' \in \cochar$.  If $\defparam - \defparam' = \phi(\rho)$ for a cocharacter $\rho$ of $Z(\bG)$, then
$$
\cM_C^\defparam(Q, \bv, \bw) \cong  \cM_C^{\defparam'}(Q, \bv, \bw) 
$$
\end{Lemma}

\begin{proof}
Define $\widetilde{\bG}_\defparam = \bG \times \CC^\times$, which acts on $\bN$ via the map $\widetilde{\bG}_\defparam \rightarrow \widetilde{\bG}$ defined by $(g, s) \mapsto (g, \defparam(s))$. At the level of $\CC$-points it is straightforward to see that $\cR_{\widetilde{\bG}_\defparam, \bN}(n) \cong \cR_{\widetilde{\bG}, \bN}(n \defparam)$, where $n\in \ZZ$ is thought of as a cocharacter of the flavour symmetry group $\CC^\times$ of $\widetilde{\bG}_\defparam$.  This extends to the diagram defining the convolution product \cite[\S 3(i)]{BFN1}, and so there is an isomorphism of varieties
$$
 \cM_C^1( \widetilde{\bG}_\defparam, \bN)  \cong \cM_C^\defparam(Q, \bv, \bw) 
$$

Now assume that $\defparam = \defparam' + \phi(\rho)$.  Then there is an isomorphism $\widetilde{\bG}_{\defparam} \cong \widetilde{\bG}_{\defparam'}$ defined by $(g, s) \mapsto (\rho(s)^{-1} g, s)$. It intertwines their actions on $\bN$, and induces the identity map on their flavour symmetry groups.  Thus $\cM_C^1( \widetilde{\bG}_\defparam, \bN) \cong \cM_C^1( \widetilde{\bG}_{\defparam'}, \bN)$. 

Altogether, we get a chain of isomorphisms which proves the claim:
$$
\cM_C^\defparam(Q, \bv,\bw) \cong  \cM_C^1( \widetilde{\bG}_\defparam, \bN) \cong  \cM_C^1( \widetilde{\bG}_{\defparam'}, \bN) \cong \cM_C^{\defparam'}(Q, \bv,\bw)
$$
\end{proof}

We will call $\defparam, \defparam'$ {\bf equivalent} if $\defparam - \defparam' = \phi(\rho)$ as in the above lemma.  It will be useful to write this notion out more explicitly. Denote the components of our coweight $\varkappa$ of $\bF$ by
\begin{equation*}
\defparam = \big( (\defparam_e)_{e\in E}, (\defparam_{i,1},\ldots, \defparam_{i, \bw_i})_{i\in I} \big),
\end{equation*}
and those of the coweight $\rho$ of $Z(\bG) \cong \prod_i \CC^\times$ by $\rho = (\rho_i)_{i\in I} \in \ZZ^I$.  From the definition of the homomorphism $\phi$, we see that $\defparam$ is equivalent to 
\begin{equation}
\label{eq: equivalence}
\defparam' = \defparam- \phi(\rho) = \big( (\defparam_e-\rho_{t(e)} + \rho_{s(e)})_{e\in E}, (\defparam_{i,1} +\rho_i,\ldots, \defparam_{i, \bw_i}+\rho_i)_{i\in I}\big)
\end{equation}

\subsection{Examples}
\label{examples}
We now recall some important examples of the varieties $\cM_C(Q, \bv, \bw)$, which have been studied by Braverman-Finkelberg-Nakajima and Nakajima-Takayama \cite{BFN2, NakTak,BFN4}.

\subsubsection{Finite ADE types}
\label{example 1}
If $Q$ is an orientation of a Dynkin diagram of finite ADE type, then 
$$\cM_C(Q, \bv, \bw) \cong \overline{\cW}^{\lambda^\ast}_{\mu^\ast}$$
is a generalized affine Grassmannian slice for the corresponding adjoint group $G_Q$ of ADE type \cite[Theorem 3.10]{BFN2}\footnote{Following \cite{BFN2} we denote $\lambda^\ast = - w_0 \lambda$, where $w_0 \in W$ is the longest element.}.   Up to equivalence as in (\ref{eq: equivalence}), we may assume that $\defparam_e = 0$ for all $e\in E$.  List the $\defparam_{i,r}$ in decreasing order:
$$
k_1 \geq k_2 \geq \ldots \geq k_N, \qquad \text{ where } N = \sum_i \bw_i
$$
If $k_a = \defparam_{i,r}$ then we denote the corresponding fundamental coweight of $G_Q$ by $\lambda_a = \varpi_i$.  We get a tuple $\underline{\lambda} = (\lambda_1,\ldots,\lambda_N)$.  Thinking of $ (k_1,\ldots,k_n)$ as a cocharacter of $(\CC^\times)^N \cong \prod_i (\CC^\times)^{\bw_i}$, by construction it lies in the cone $\defparam \in \Lambda_F^{++}$ from \cite[\S 5(ii)]{BFN4}.  Thus by \cite[\S 5(v)]{BFN4}, if $k_1> \ldots > k_N$ are distinct (i.e.~if the $\varkappa_{i,r}$ are distinct), then
\begin{equation*}
\cM^\defparam_C(Q, \bv, \bw) \cong \widetilde{\cW}^{\underline{\lambda}^\ast}_{\mu^\ast}
\end{equation*}
where the right-hand side is defined as in \cite[\S 5(i)]{BFN4}.  If the $\lambda_a$ are all minuscule, then this variety is smooth and symplectic. 
This can be proven similarly to the main theorem of \cite{MuthiahWeekes}, cf.~also \cite[Theorem 2.9]{KWWY}.  

In particular, in finite type A, the variety $\cM_C^\defparam(Q, \bv, \bw)$ is smooth and symplectic so long as  $\defparam$ does not lie on a finite collection of hyperplanes.  This clearly remains true even if we do not impose that $\defparam_e = 0$.

\subsubsection{Affine type A}
\label{example 2}
In affine type A, $\cM_C(Q, \bv, \bw)$ is isomorphic to a bow variety by \cite[Theorem 6.18]{NakTak}.  Label the vertices by $I = \ZZ / n \ZZ$, and pick the orientation $Q$ where $i\rightarrow i+1$ for $0\leq i <n$.  Up to equivalence, we may assume that $\defparam_e = 0$ except for $e$ the arrow $n-1 \rightarrow 0$.  Then for $\defparam$ satisfying the inequality \cite[(4.4)]{BFN4}, the variety $\cM_C^\defparam(Q, \bv, \bw)$ is identified with a partially resolved bow variety. For generic $\defparam$ this variety is smooth and symplectic, see \cite[\S 6.2]{NakTak}.

\subsection{Hyperplane arrangement}
Our next task will be to understand the pairs $(Z_\bG(t), \bN^t)$ that can arise for quiver gauge theories,  where $t \in \ft$.  First some terminology: we call $(Q, \bv, \bw)$ {\bf connected} if $\{i \in I:\bv_i \neq 0\}$ is a connected subset of $Q$, and if $\bw_i \neq 0$ implies $\bv_i \neq 0$ for all $i\in I$.  \footnote{Note that the latter condition is very mild: if for some $i\in I$ we have  $\bw_i \neq 0$ but $\bv_i = 0$, then there is no corresponding summand $\Hom(W_i, V_i)$ in (\ref{groups from quiver data}).  Thus in this case we obtain the same $(\bG, \bN)$ by setting $\bw_i = 0$.}

For a given $(Q,\bv,\bw)$ and $t\in \ft$, we'll denote $\codim_{Q,\bv, \bw} (t) \DEF \codim_{\bG, \bN}(t)$ where the latter is defined as in (\ref{eq: codim function}).

\begin{Proposition}
\label{lemma: components}
Let $(Q, \bv, \bw)$ be fixed, with corresponding $(\bG,\bN)$. For any $t \in \ft$ there exists a decomposition
$$
\cM^\defparam_C(Z_\bG(t), \bN^t)\ \cong\ \cM^\defparam_C(Q,\bv^{(1)}, \bw^{(1)}) \times \cdots \times \cM^\defparam_C(Q, \bv^{(m)}, \bw^{(m)}),
$$
which holds for all $\defparam \in \cochar$, and where the data $(Q, \bv^{(\ell)}, \bw^{(\ell)})$ are all connected.  Moreover, 
\begin{enumerate}[(i)]
\item $\bv = \sum_\ell \bv^{(\ell)}$,
\item  for each $i \in I$ there is at most one $\ell$ with $\bw_i^{(\ell)} \neq 0$, and in this case $\bw_i^{(\ell)} = \bw_i$,
\item for each $\ell$, the restriction map on flavour symmetry groups
$$\bF(Q, \bv, \bw) \twoheadrightarrow \bF(Q, \bv^{(\ell)}, \bw^{(\ell)})$$ 
is the obvious surjection respecting the factors in (\ref{eq: flavour symmetry group}),
\item $ \codim_{Q,\bv, \bw} (t) = \sum_\ell \codim_{Q,\bv^{(\ell)}, \bw^{(\ell)}}(0)$
\end{enumerate}
\end{Proposition}
\begin{proof}
We may think of $t \in \operatorname{Lie}\bG = \bigoplus_i \End(V_i)$.  Denote the $\lambda$--eigenspace of $t$ on $V_i$ by $V_i(\lambda)$.  Then we can identify
\begin{align*}
Z_\bG(t) & = \prod_\lambda  \prod_i \GL(V_i(\lambda)), \\
\bN^t & = \Big(\bigoplus_\lambda \bigoplus_{e \in E} \Hom(V_{s(e)}(\lambda), V_{t(e)}(\lambda) )\Big) \oplus \bigoplus_{i \in I} \Hom( W_i, V_i(0))
\end{align*}
For each $\lambda$, consider the subgroup $\bG_\lambda \subset Z_\bG(t)$ corresponding to factors labelled by $\lambda$ above, and $\bN_\lambda \subset \bN^t$ corresponding to the summands labelled by $\lambda$. (In particular, the framing $W$ only contributes $\bN_\lambda$ for $\lambda = 0$.) Then $Z_\bG(t) = \prod_\lambda \bG_\lambda$ and $\bN^t = \bigoplus_\lambda \bN_\lambda$, with $\bG_\lambda$ acting trivially on $\bN_\mu$ unless $\lambda = \mu$. 

By construction each $(\bG_\lambda,\bN_\lambda)$ corresponds to a quiver gauge theory datum $(Q,\bv_\lambda,\bw_\lambda)$, but might not be connected.  So we will make a further refinement. First, if for some $\lambda, i$ we have $(\bw_\lambda)_i \neq 0$ and $(\bv_\lambda)_i = 0$, then we may set $(\bw_\lambda)_i = 0$ by Remark \ref{remark: no summand}.  (Note that this can only happen for $\lambda = 0$.) Next, for each connected component $X$ of $\{i \in I : (\bv_\lambda)_i \neq 0\}$ there is a corresponding $\bG_\lambda^X = \prod_{i\in X} \GL(V_i(\lambda)) \subset \bG_\lambda$ and $\bN_\lambda^X \subset \bN_\lambda$. Then $Z_\bG(t) = \prod_{\lambda, X} \bG_\lambda^X$ and $\bN^t = \bigoplus_{\lambda,X} \bN_\lambda^X$.  By construction, each $(\bG_\lambda^X, \bN_\lambda^X)$ corresponds to a connected quiver datum, as desired.

Only finitely many pairs $(\lambda, X)$ ever contribute non-trivially, so we may relabel the above decompositions simply by $Z_\bG(t) = \prod_{\ell=1}^m \bG^{(\ell)}$ and $\bN^t = \bigoplus_{\ell=1}^m \bN^{(\ell)}$. Each $(\bG^{(\ell)}, \bN^{(\ell)})$ corresponds to some connected quiver datum $(Q, \bv^{(\ell)}, \bw^{(\ell)})$.  Tracing through the construction, we see that the group $\bF$ respects these decompositions.  Therefore Lemma \ref{lemma: segre} applies, which shows that
$$
\cM^\defparam_C(Z_\bG(t), \bN^t)\ \cong\ \cM^\defparam_C(Q,\bv^{(1)}, \bw^{(1)}) \times \cdots \times \cM^\defparam_C(Q, \bv^{(m)}, \bw^{(m)})
$$
Parts (i), (ii), (iii) follow easily from the above construction.  Finally, the formula (iv) follows from  (\ref{eq: codim = codim}) together with our decomposition of $(Z_\bG(t), \bN^t)$.
\end{proof}

Next, we classify quivers with $\codim_{Q, \bv, \bw} (0) \leq 3$. 
\begin{Lemma}
\label{lemma: classification}
Let $(Q, \bv, \bw)$ be connected.  Then
$$
\codim_{Q,\bv, \bw}(0) = \left\{ \begin{array}{cl} \sum_i \bv_i, & \text{if some } \bw_i \neq 0 \\ \sum_i \bv_i - 1, & \text{if all } \bw_i = 0 \end{array} \right.
$$
In particular, Figure \ref{figure} gives a complete list of connected $(Q,\bv,\bw)$ with $\codim_{Q,\bv,\bw}(0) \leq 3$.
\end{Lemma}
This dichotomy corresponds to whether $\bG$ acts on $\bN$ faithfully (some $\bw_i \neq 0$) or not (all $\bw_i = 0 $).
\begin{proof}
Let $x\in \ft$.  If all generalized roots vanish at $x$, then this forces all components of $x$ to be equal (with respect to the standard bases of the $V_i = \CC^{\bv_i}$).  If there is no framing then there is no further constraint, so the vanishing locus is 1-dimensional.  If there is framing, then this further forces $x=0$. This proves the formula for the codimension.

The classification in Figure \ref{figure} is now straightforward: if there is no framing then we demand $\sum_i \bv_i -1 \leq 3$, so in particular there are at most 4 vertices.  If there is framing, then we instead demand $\sum_i \bv_i \leq 3$, so there are at most 3 vertices.
\end{proof}

\begin{figure}
$$
\begin{tikzcd}[sep=small,cells={nodes={draw=black, ellipse,anchor=center,minimum height=2em}}]
|[draw=black, rectangle]| n \ar[d] \\
1
\end{tikzcd}
\qquad
\begin{tikzcd}[sep=small,cells={nodes={draw=black, ellipse,anchor=center,minimum height=2em}}]
|[draw=black, rectangle]| n \ar[d] \\
2
\end{tikzcd}
\qquad
\begin{tikzcd}[sep=small,cells={nodes={draw=black, ellipse,anchor=center,minimum height=2em}}]
|[draw=black, rectangle]| n \ar[d] \\
3
\end{tikzcd}
\qquad
\begin{tikzcd}[cells={nodes={draw=black, ellipse,anchor=center,minimum height=2em}}]
4
\end{tikzcd}
\qquad
\begin{tikzcd}[sep=small,cells={nodes={draw=black, ellipse,anchor=center,minimum height=2em}}]
|[draw=black, rectangle]| n \ar[d] & |[draw=black, rectangle]| m \ar[d] \\
1 \ar[r] & 1
\end{tikzcd}
\qquad
\begin{tikzcd}[sep=small,cells={nodes={draw=black, ellipse,anchor=center,minimum height=2em}}]
|[draw=black, rectangle]| n \ar[d] & |[draw=black, rectangle]| m \ar[d] \\
2 \ar[r] & 1 
\end{tikzcd}
\qquad
\begin{tikzcd}[sep=small,cells={nodes={draw=black, ellipse,anchor=center,minimum height=2em}}]
2 \ar[r] & 2
\end{tikzcd}
$$
\smallskip
$$
\begin{tikzcd}[sep=small,cells={nodes={draw=black, ellipse,anchor=center,minimum height=2em}}]
3 \ar[r] & 1
\end{tikzcd}
\qquad
\begin{tikzcd}[sep=small,cells={nodes={draw=black, ellipse,anchor=center,minimum height=2em}}]
 |[draw=black, rectangle]| n \ar[d] & |[draw=black, rectangle]| m \ar[d]  & |[draw=black, rectangle]| p \ar[d] \\
1 \ar[r]& 1 \ar[r] & 1
\end{tikzcd}
\qquad
\begin{tikzcd}[sep=tiny,cells={nodes={draw=black, ellipse,anchor=center,minimum height=2em}}]
& &  |[draw=black, rectangle]|n \ar[dd]& &\\
& & & &\\
& &1  \ar[dr]& &\\
& 1 \ar[ur] & & 1 \ar[ll]&\\
|[draw=black, rectangle]|m \ar[ur]  & & & & |[draw=black, rectangle]|p \ar[ul]
\end{tikzcd}
\begin{tikzcd}[sep=small,cells={nodes={draw=black, ellipse,anchor=center,minimum height=2em}}]
2 \ar[r] & 1 \ar[r] & 1
\end{tikzcd}
$$
\smallskip
$$
\begin{tikzcd}[sep=small,cells={nodes={draw=black, ellipse,anchor=center,minimum height=2em}}]
1 \ar[r] & 2 \ar[r] & 1 
\end{tikzcd}
\qquad
\begin{tikzcd}[sep=tiny,cells={nodes={draw=black, ellipse,anchor=center,minimum height=2em}}]
&2  \ar[dr]& \\
1 \ar[ur] & & 1  \ar[ll]
\end{tikzcd}
\qquad
\begin{tikzcd}[sep=small,cells={nodes={draw=black, ellipse,anchor=center,minimum height=2em}}]
1 \ar[r]& 1 \ar[r] & 1 \ar[r]  & 1
\end{tikzcd}
$$
\medskip
$$
\begin{tikzcd}[sep=tiny,cells={nodes={draw=black, ellipse,anchor=center,minimum height=2em}}]
& 1 \ar[dd]&\\
& & \\
&1  \ar[dr]& \\
1 \ar[ur] & & 1
\end{tikzcd}
\qquad
\begin{tikzcd}[sep=tiny,cells={nodes={draw=black, ellipse,anchor=center,minimum height=2em}}]
& 1 \ar[dd]&\\
& & \\
&1  \ar[dr]& \\
1 \ar[ur] & & 1 \ar[ll]
\end{tikzcd}
\qquad
\begin{tikzcd}[sep=small,cells={nodes={draw=black, ellipse,anchor=center,minimum height=2em}}]
1 \ar[r] \ar[d] & 1 \ar[d] \\
1 \ar[r] & 1 
\end{tikzcd}
\qquad
\begin{tikzcd}[sep=small,cells={nodes={draw=black, ellipse,anchor=center,minimum height=2em}}]
1 \ar[r] \ar[d] \ar[dr] & 1 \ar[d] \\
1 \ar[r] & 1 
\end{tikzcd}
\qquad
\begin{tikzcd}[sep=small,cells={nodes={draw=black, ellipse,anchor=center,minimum height=2em}}]
1 \ar[r]  \ar[d]\ar[dr] & 1 \ar[d] \\
1 \ar[r] \arrow[ur, crossing over] & 1 
\end{tikzcd}
$$
\caption{Up to orientation, the list of all connected quiver gauge theory data with $\codim_{Q,\bv, \bw}(0) \leq 3$.    The framing dimensions $n,m,p \geq 0$.  Note that we consider only simple quivers.}
\label{figure}
\end{figure}
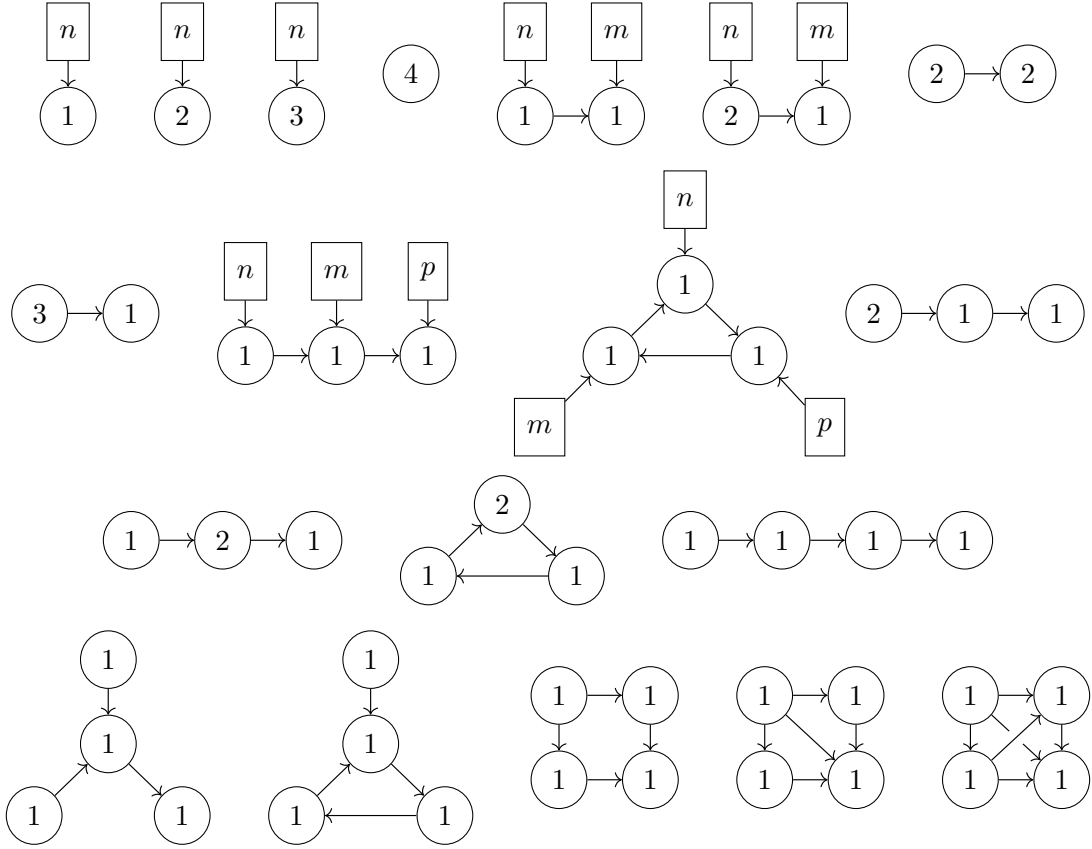

\subsection{The proof of Theorem \ref{main thm}}
\label{section: proof}

We will apply Theorem \ref{criterion}.   Recall that $t \in \ft \setminus \ft^{(4)}$ means that $\codim_{Q,\bv,\bw}(t) \leq 3$.  By combining Proposition \ref{lemma: components} and Lemma \ref{lemma: classification},  we see that for $t \in \ft \setminus \ft^{(4)}$
\begin{equation}
\label{eq: decomp main proof}
\cM^\defparam_C(Z_\bG(t), \bN^t) \cong \cM^\defparam_C(Q,\bv^{(1)}, \bw^{(1)}) \times \cdots \times \cM^\defparam_C(Q, \bv^{(m)}, \bw^{(m)})
\end{equation}
where all factors on the right hand side are from Figure \ref{figure}, and where 
$$
\codim_{Q,\bv,\bw}(t) = \sum_\ell \codim_{Q, \bv^{(\ell)}, \bw^{(\ell)}}(0) \leq 3
$$
We claim that any particular factor $\cM_C^\defparam(Q, \bv^{(\ell)}, \bw^{(\ell)})$ above is smooth so long as $\defparam$ avoids a finite set of hyperplanes.  Assuming this claim for the moment, this implies that $\cM^\defparam_C(Z_\bG(t), \bN^t)$ is smooth so long as $\defparam$ avoids the (still finite) union over $\ell$ of these sets of  hyperplanes.  Note also that as we vary $t\in \ft \setminus \ft^{(4)}$, only finitely many decompositions (\ref{eq: decomp main proof}) can arise.   Therefore for a generic choice of $\defparam$, the space $\cM_C^\defparam(Z_\bG(t), \bN^t)_t$ is smooth for all $t\in\ft\setminus \ft^{(4)}$.  So $\cM_C(Q, \bv, \bw)$ has symplectic singularities by  Theorem \ref{criterion}.

To prove the claim, by part (iii) of Proposition \ref{lemma: components} we can reduce to the case where $(Q, \bv, \bw)$ is one of the quivers in Figure \ref{figure}, with the intrinsic flavour symmetry group $\bF = \bF(Q, \bv, \bw)$.

Note that many of the quivers in Figure \ref{figure} have all $\bv_i =1$, and so correspond to toric hyperk\"ahler varieties \cite[\S 4(vii)]{BFN1}. The flavour symmetry $\bF$ for the quiver (defined as in (\ref{eq: flavour symmetry group})) surjects onto the full Abelian flavour symmetry group: $\bF$ contains subtori which dilate each edge independently, with weight 1.  Thus we can appeal to the unimodularity criterion for existence of resolutions given in \cite[Proposition 4.11]{BellamyKuwabara}.  This is easy to verify, and proves the claim in these cases.

For the remaining quivers in Figure \ref{figure} where some $\bv_i \geq 2$, all but one (see below) are of finite type A, so the claim in these cases follows from the results of \cite[\S 5]{BFN4} discussed in \S \ref{example 1}.  This leaves one quiver which is of affine type A, as discussed in \S \ref{example 2}.  Thus the claim holds in all cases, which completes the proof of the theorem.
%

\begin{Remark}
\label{remark: qfactterm}
We also conclude from Theorem \ref{criterion} that $\cM_C^\defparam(Q, \bv, \bw)$ has terminal singularities, for generic $\defparam$.  We optimistically conjecture that it is a $\QQ$--factorial terminalization of $\cM_C(Q,\bv,\bw)$.
\end{Remark}

\begin{Remark}
\label{rem: quiver manipulations}
The affine type A quiver discussed in the proof above is also isomorphic to a product of finite type A cases:
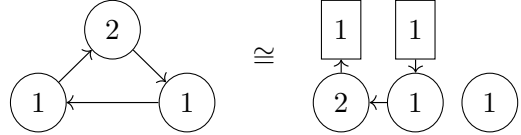
\begin{equation*}
\label{eq: isom of quiver data}
\begin{tikzcd}[sep=tiny,cells={nodes={draw=black, ellipse,anchor=center,minimum height=2em}}]
&2  \ar[dr]& \\
1 \ar[ur] & & 1  \ar[ll]
\end{tikzcd}
\quad 
\cong 
\quad
\begin{tikzcd}[sep=tiny,cells={nodes={draw=black, ellipse,anchor=center,minimum height=2em}}]
|[draw=black, rectangle]| 1  & |[draw=black, rectangle]| 1 \ar[d] & \\
2 \ar[u] & 1 \ar[l] & 1 
\end{tikzcd}
\end{equation*}
More precisely, the left side has $\bG = \GL(2) \times \CC^\times \times \CC^\times$ with factors corresponding to the vertices in clockwise order.  Consider the automorphism of $\bG$ defined by $(g, s, t) \mapsto (g s^{-1} t, s^{-1} t, t)$. The pull-back of the left side quiver datum under this isomorphism is the right side quiver datum. This map also induces an isomorphism on flavour symmetry. 
\end{Remark}

Using a similar argument, we also obtain the following generalization of Theorem \ref{thm: BFN theorem}:
\begin{Theorem}
\label{thm: BFN theorem for resolutions}
For any $\defparam \in \cochar$, the variety $\cM_C^\defparam(Q,\bv,\bw)$ is irreducible, normal, and its Poisson structure is symplectic on its smooth locus.
\end{Theorem}
\begin{proof}
Note that irreducibility follows from Lemma \ref{cor: integral}. The proof of the remaining properties is similar to the proof of Theorem \ref{section: proof} above, but appealing now to Theorem \ref{criterion: normal symplectic}. We must study $t \in \ft \setminus \ft^{(2)}$, and show that in the product decomposition from Proposition \ref{lemma: components}, all factors $\cM_C^\defparam(Q,\bv^{(\ell)}, \bw^{(\ell)})$ are normal and symplectic on their smooth loci. Recall that $t \in \ft \setminus \ft^{(2)}$ means that $\codim_{Q,\bv,\bw}(t) =  \sum_\ell \codim_{Q, \bv^{(\ell)}, \bw^{(\ell)}}(0)\leq 1$.  By Lemma \ref{lemma: classification}, we see that all $(Q,\bv^{(\ell)},\bw^{(\ell)})$ belong to the following short list:
$$
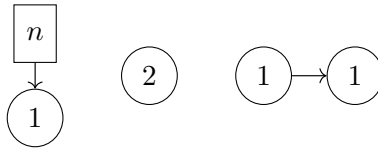
\begin{tikzcd}[sep=small,cells={nodes={draw=black, ellipse,anchor=center,minimum height=2em}}]
|[draw=black, rectangle]| n \ar[d] \\
1
\end{tikzcd}
\qquad
\begin{tikzcd}[sep=small,cells={nodes={draw=black, ellipse,anchor=center,minimum height=2em}}]
2
\end{tikzcd}
\qquad
\begin{tikzcd}[sep=small,cells={nodes={draw=black, ellipse,anchor=center,minimum height=2em}}]
1 \ar[r] & 1
\end{tikzcd}
$$
Thus all $\cM_C^\defparam(Q,\bv^{(\ell)},\bw^{(\ell)})$ are partially resolved bow varieties,  as discussed in \S \ref{example 2}.  Thus they are normal, and symplectic on their smooth loci, by the results of \cite{NakTak}.  (One can also deduce these properties from the finite type A perspective of \S \ref{example 1}.) This completes the proof.
\end{proof}

\subsection{Extension to all quivers}
\label{section: extending}

It is natural to ask whether the proofs given above extend to all quiver gauge theories, allowing loops and/or multiple edges.  The required steps are simple to state: first, we must allow multiple edges and/or loops wherever possible for all quivers in Figure \ref{figure}.  For each resulting quiver, we should show that $\cM_C^\defparam(\bv, \bw)$ is smooth and symplectic for generic $\defparam$ (or at least that it is normal, symplectic on its smooth locus, and has singular locus of codimension $\geq 4$). 

By a similar argument to Remark \ref{rem: quiver manipulations}, we can rewrite each of these new quivers in terms of products of simpler ones: toric hyperk\"ahler varieties (which we omit), or one of

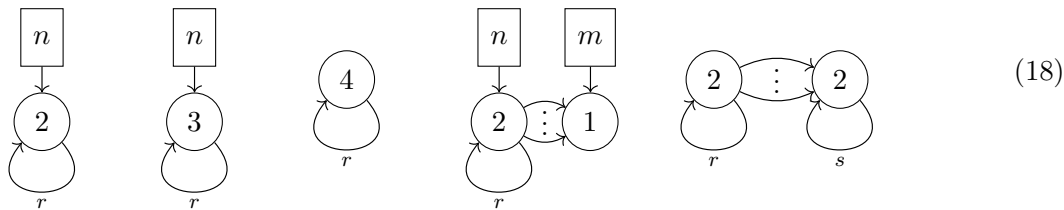
\begin{equation}
\begin{tikzcd}[sep=small,cells={nodes={draw=black, ellipse,anchor=center,minimum height=2em}}]
|[draw=black, rectangle]| n \ar[d] \\
2 \arrow[out=-50,in=-130,loop,distance = 3em, start anchor={south east}, end anchor = {south west}, "r"]
\end{tikzcd}
\begin{tikzcd}[sep=small,cells={nodes={draw=black, ellipse,anchor=center,minimum height=2em}}]
|[draw=black, rectangle]| n \ar[d] \\
3 \arrow[out=-50,in=-130,loop,distance = 3em, start anchor={south east}, end anchor = {south west}, "r"]
\end{tikzcd}
\begin{tikzcd}[sep=small,cells={nodes={draw=black, ellipse,anchor=center,minimum height=2em}}]
4 \arrow[out=-50,in=-130,loop,distance = 3em, start anchor={south east}, end anchor = {south west}, "r"]
\end{tikzcd}
\begin{tikzcd}[sep=small,cells={nodes={draw=black, ellipse,anchor=center,minimum height=2em}}]
|[draw=black, rectangle]| n \ar[d] & |[draw=black, rectangle]| m \ar[d] \\
2 \arrow[out=-50,in=-130,loop,distance = 3em, start anchor={south east}, end anchor = {south west}, "r"]
\ar[r,bend left =30] \ar[r, bend right=30] \arrow[r,phantom,shift left =0.5ex,"\vdots"] & 1 
\end{tikzcd}
\ \
\begin{tikzcd}[cells={nodes={draw=black, ellipse,anchor=center,minimum height=2em}}]
2 \arrow[out=-50,in=-130,loop,distance = 3em, start anchor={south east}, end anchor = {south west}, "r"]
\ar[r,bend left =25] \ar[r, bend right=25] \arrow[r,phantom,shift left =0.5ex,"\vdots"]& 
2 \arrow[out=-50,in=-130,loop,distance = 3em, start anchor={south east}, end anchor = {south west}, "s"]
\end{tikzcd}
\end{equation}
Dots denote multiple edges, $r,s \geq 0$ indicate multiple loops, and the framing dimensions $n, m \geq 0$.  Proving appropriate smoothness for the above quivers would extend our result to all quivers.

\begin{Remark}
The Coulomb branches for the first and second ``multiloop'' quivers above are denoted $\cM_C(r,2,n)$ and $\cM_C(r,3,n)$ in \cite{FinkGon}.  There, Finkelberg and Goncharov conjecture that these varieties all admit symplectic resolutions, and prove this conjecture for the cases $\cM_C(r,2,1)$.
\end{Remark}

\bibliographystyle{myamsalpha}
\bibliography{symmtype}
\end{document}